\documentclass[11pt]{article}
\usepackage[a4paper,margin=1in]{geometry}
\usepackage[T1]{fontenc}
\usepackage[utf8]{inputenc}
\usepackage{microtype}
\usepackage{setspace}
\usepackage{parskip}
\usepackage{amsmath,amssymb,amsthm,amsfonts,mathtools}
\usepackage{mathrsfs}
\usepackage{bm}
\usepackage[all]{xy}
\usepackage{graphicx}
\usepackage{tikz}
\usepackage{enumitem}
\usepackage{booktabs}
\usepackage{array}
\usepackage[nameinlink,noabbrev]{cleveref}
\usepackage{mathpazo}

\usepackage{titlesec}

\titleformat{\section}
  {\large\bfseries}
  {\thesection.}{0.5em}{}
\titleformat{\subsection}
  {\normalsize\bfseries}
  {\thesubsection.}{0.5em}{}
\titleformat{\subsubsection}
  {\normalsize\itshape}
  {\thesubsubsection.}{0.5em}{}

\newtheoremstyle{pretty}
  {1em}{1em}
  {\itshape}
  {}
  {\bfseries}
  {.}
  {0.5em}
  {}
\newtheoremstyle{plainnote}
  {1em}{1em}
  {\normalfont}
  {}
  {\bfseries}
  {.}
  {0.5em}
  {}

\theoremstyle{pretty}
\newtheorem{theorem}{Theorem}[section]
\newtheorem{lemma}[theorem]{Lemma}
\newtheorem{proposition}[theorem]{Proposition}
\newtheorem{corollary}[theorem]{Corollary}

\theoremstyle{plainnote}

\newtheorem{remark}[theorem]{Remark}

\newcommand{\C}{\mathbb{C}}
\renewcommand{\2}{\mathbf{2}}

\newcommand{\V}{\mathcal{V}}

\newcommand{\Hoop}{\mathbf{Hp}}
\newcommand{\HeytingSemilattices}{\mathbf{HSL}}

\newcommand{\BoundedHoop}{\mathbf{BHp}}
\newcommand{\BoundedHeytingSemilattices}{\mathbf{BHSL}}

\title{
  \vspace{-2em}
  {\Huge\bfseries A note on coextensivity of bounded hoops}\\[1em]
  \rule{0.75\textwidth}{0.6pt}
}
\author{Michael Hoefnagel,  Danielle Kleyn, Giuseppe Metere} 
\date{\today}

\begin{document}

\maketitle

\begin{abstract}
The main aim of this note is to provide a characterisation of coextensive morphisms in the category of bounded hoops. This characterisation is then used to show that several categories of bounded hoops are coextensive as categories. Among these are the variety of bounded Wajsberg hoops, or more generally the variety of bounded $\vee$-hoops. Our characterisation also recovers the known coextensivity of the category $\mathbf{Heyt}$ of Heyting algebras and the category $\mathbf{MV}$ of MV-algebras.
\end{abstract}



\section{Introduction}
Hoops were introduced by Bosbach under the name \emph{complementary semigroups} \cite{Bosbach1969}. We use the axiomatization of \cite{BlokFerreirim2000StructureHoops}: a \emph{hoop} is an algebra $\mathbf{H}=(H,\cdot,\to,1)$ such that $(H,\cdot,1)$ is a commutative monoid and for all $x,y,z\in H$,
\begin{align*}
    x \to x &= 1,\\
      x \to (y \to z) &= (x\cdot y)\to z,\\
      (x\to y)\cdot x &= (y\to x)\cdot y.
    \end{align*}
Given a hoop $\mathbf{H}=(H,\cdot,\to,1)$, there is a partial order on $H$ defined by
\[
x \leqslant y \iff x \to y = 1.
\]
This relation is a partial order on $H$, called the \emph{natural order} of the hoop. With respect to this ordering, every hoop admits meets, which are equationally defined by
\[
x \wedge y = (y\to x)\cdot y.
\]
Not all hoops admit joins; however, the term 
\[
x \sqcup y = \big ((x \to y) \to y \big ) \wedge \big ((y \to x) \to x \big )
\]
defines what is called the \emph{pseudojoin}. A hoop is called a $\vee$-hoop if the pseudojoin represents the join with respect to the natural ordering. Note that the full subcategory $\Hoop_\vee$ of $\vee$-hoops may be presented as the subvariety of $\Hoop$ satisfying the equation
\[
\big ((x \to y) \to z \big ) \wedge \big ((y \to x) \to z \big ) = z.
\]
A hoop is called \emph{bounded} if it admits a bottom element with respect to the natural order. In what follows, we will denote the category of bounded hoops by $\BoundedHoop$, where morphisms are morphisms of hoops which preserve the bottom element. Bounded hoops may also be presented as a variety, where one appends the constant $0$ to the signature of $\Hoop$ together with the equation $0 \to x = 1$. As usual, we write $\neg x=x\to 0$. The canonical two-element algebra $\2 = \{0,1\}$ is the initial object in $\BoundedHoop$. Finally, let us denote by $\BoundedHoop_\vee$ the category of bounded hoops in $\Hoop_\vee$.

\subsection{Idempotents and canonical quotients}
In a bounded hoop $X$, two elements $e_1, e_2$ are called \emph{weakly complementary central elements}  if there exists a morphism $f:\2 \times \2 \to X$ with 
\[
e_1 = f(1,0) \quad \text{ and } \quad e_2 = f(0,1). 
\]
In the terminology of \cite{Botero2021}, complementary central pairs are
those pairs associated with product decompositions. The adjective
``weakly'' indicates that the condition above alone need not yield a
product decomposition; Lemma~\ref{lemma: weakly complementary central}
identifies precisely when it does.

Note that two such elements $e_1$ and $e_2$ are idempotent and we have 
\begin{align*}
 e_1 \to e_2 = e_2, \quad 
 e_2 \to e_1 = e_1,  \quad
 e_1 \to 0 = e_2, \quad
 e_2 \to 0 = e_1, \quad
 e_1 \sqcup e_2 = 1,  \quad
 e_1 e_2 = 0. 
\end{align*}
Given any element $e$ in a (bounded) hoop $X$, the set $e \to X$ is the set 
\[
\{e \to x \mid x \in X\}.
\] 
For an idempotent $e$, the homomorphism property of \cite{VeroffSpinks2004} will show that $e\to X$ is a subhoop of $X$.

Recall the following elementary fact (\cite[Equation~(61)]{VeroffSpinks2004}), whose proof we include for completeness.
\begin{lemma} \label{lemma: idempotent}
Let $e$ be any idempotent in a hoop $X$, then $e\wedge x = e\cdot x$ and $e \wedge (-)$ is left adjoint to $e \to (-)$.
\end{lemma}
\begin{proof}
We always have $e\cdot x\leqslant e \wedge x$ and for the reverse we have $e\wedge x = e\cdot (e\to x)$ so that $e\cdot (e\wedge x) = e \wedge x$ and hence 
\[
e \wedge x \leqslant x \implies e\cdot (e \wedge x) \leqslant e\cdot x \implies e \wedge x \leqslant e\cdot x.  
\]
Therefore, with respect to $e$, the residuation property 
\[
x \wedge e \leqslant y \quad \text{iff} \quad x \leqslant e \to y
\]
expresses that $e \to (-)$ is right adjoint to $(-) \wedge e$. 
\end{proof}

We now associate with every idempotent $e$ a canonical quotient hoop $[e\to X]$.
\begin{proposition}\label{prop:filter_is_hoop}
Given any idempotent $e$ in a (bounded) hoop $X$, the set $e \to X$ is a (bounded) hoop and $\rho^X_e:X \to [e \to X]$ defined by $\rho^X_e(x) = e\to x$ is a morphism of (bounded) hoops. 
\end{proposition}
\begin{proof}
For any idempotent $e$, Equation~(1) in \cite{VeroffSpinks2004} and  the main theorem therein give, respectively,
\begin{align*}
(e\to x)\to(e\to y)&=e\to(x\to y),\\
(e\to x)\cdot (e\to y)&=e\to(x\cdot y).
\end{align*}
Thus $e\to X$ is closed under implication and multiplication, and $\rho_e^X$ preserves both operations.
Finally, $\rho_e^X(1) = 1$ and if $X$ has a bottom element, then $\rho^X_e(0)$ is the bottom element of $e \to X$, since $\rho^X_e$ is surjective and monotone. 
\end{proof}

\begin{lemma} \label{lemma: weakly complementary central}
Let $X$ be any bounded hoop and $e_1, e_2$ any two weakly complementary central elements. Then the induced map 
\[
X \to [e_1 \to X] \times [e_2 \to X]
\]
is surjective. It is injective if and only if $e_1$ and $e_2$ admit a join and that join is $1$.
\end{lemma}
\begin{proof}
By Lemma~\ref{lemma: idempotent}, the map $x \mapsto e\to x$ preserves finite meets, since it is a right adjoint.
For surjectivity, note that for any pair $(e_1 \to x, e_2 \to y)$ we may define $z = (e_1 \to x) \wedge (e_2 \to y)$ so that 
\begin{align*}
    e_1 \to z &= e_1 \to \big ((e_1 \to x) \wedge (e_2 \to y)\big) \\ 
    &= (e_1 \to (e_1 \to x))\wedge (e_1 \to (e_2 \to y)) \\
    &= (e_1 \to x) \wedge (e_1\cdot e_2 \to y) \\
    &= (e_1 \to x) \wedge (0 \to y) = e_1 \to x
\end{align*}
and similarly $e_2 \to z = e_2 \to y$, hence $(e_1 \to z, e_2 \to z) = (e_1 \to x, e_2 \to y)$. 

For injectivity, if $e_1 \vee e_2 = 1$ then $(e_1 \to x, e_2 \to x) = (e_1 \to y, e_2 \to y)$ implies that 
\begin{align*}
(e_i \to x) \to (e_i \to y) = 1 \implies e_i \to (x \to y) = 1 \implies e_i \leqslant x \to y
\end{align*}
so that $1 = e_1 \vee e_2 \leqslant x \to y$, which implies $x \leqslant y$ and similarly $y \leqslant x$. On the other hand, if $x \mapsto (e_1 \to x, e_2 \to x)$ is injective, and hence an isomorphism, then note that under this map $e_1$ is sent to $(1, e_1)$ and $e_2$ is sent to $(e_2, 1)$, whose join is $(1,1)$ and hence $e_1 \vee e_2 = 1$. 
\end{proof}

Thus a weakly complementary central pair is a complementary central pair
in the sense of \cite{Botero2021} precisely when $e_1\vee e_2=1$.

\begin{lemma} \label{lemma: idempotent pushout}
Let $f \colon A \times B \to X$ be any homomorphism in $\BoundedHoop$ and suppose $e = f(1,0)$. Then
\[
\xymatrix{
A \times B \ar[d]_f \ar[r]^{\pi_1} & A \ar[d]^{\bar f} \\
X \ar[r]_-{\rho^X_e} & [e\to X]
}
\]
where $\bar f(a)=f(a,1)$, is a pushout.
\end{lemma}

\begin{proof}
First note that $\ker(\rho_e^X)$ is the congruence generated by $(e,1)$. Indeed, $\rho_e^X(e)=\rho_e^X(1)=1$, while, if $e\to x=e\to y$, then
\[
(e,1)\to(x,x)=(e\to x,x)
\quad\text{and}\quad
(e,1)\to(y,y)=(e\to y,y)
\]
show that $(x,y)$ belongs to the congruence generated by $(e,1)$.

Let $(p_1,p_2)$ be the kernel pair of $\pi_1$. The pair $((1,0),(1,1))$ belongs to this kernel pair and its image under $(fp_1,fp_2)$ is $(e,1)$. Conversely, if $\pi_1(a,b)=\pi_1(a,b')$, then
\[
\rho_e^X f(a,b)=f\big((1,0)\to(a,b)\big)=f(a,1)=\rho_e^X f(a,b').
\]
Consequently, the congruence generated by the pairs in the image of $(fp_1,fp_2)$ is precisely $\ker(\rho_e^X)$. Thus $\rho_e^X$ is the coequaliser of $fp_1$ and $fp_2$, which gives the required pushout.
\end{proof}

\section{Coextensive categories of bounded hoops}
Recall that for any category $\C$ with finite products which admits pushouts along product projections, the canonical functor 
\[
(X_1\downarrow \C) \times (X_2 \downarrow \C) \to (X_1 \times X_2 \downarrow \C) \tag{$*$}
\]
is right adjoint to the functor obtained from pushing out a morphism $X_1 \times X_2 \to A$ along the canonical product projections $X_1 \leftarrow X_1 \times X_2 \to X_2$. Dually to the notion of an extensive category \cite[Proposition~2.2]{CarboniLackWalters1993}, the category $\C$ is called \emph{coextensive} if this functor is an equivalence of categories. Alternatively, coextensivity may be formulated internally with respect to its morphisms. A morphism $f:A \to X$ is called \emph{coextensive} \cite[Section~3]{HoefnagelTheart2025} if it satisfies the following two conditions: it admits pushouts along the projections of every product diagram $A_1\leftarrow A\to A_2$, and the resulting pushouts form a product diagram; and in every commutative diagram
\[\xymatrix{
      A_1  \ar[d]_{} &  A \ar[d]^f \ar[r] \ar[l] & A_2 \ar[d]\\
      X_1  & X \ar[r] \ar[l] & X_2
}\]
where both rows are product diagrams, the squares are pushouts. Then $\C$ is \emph{coextensive} if it has finite products and every morphism in $\C$ is coextensive.

The variety $\BoundedHoop$ is not coextensive. Indeed, identifying the elements $e_1 = (1,0)$ and $e_2 = (0,1)$, consider the embedding
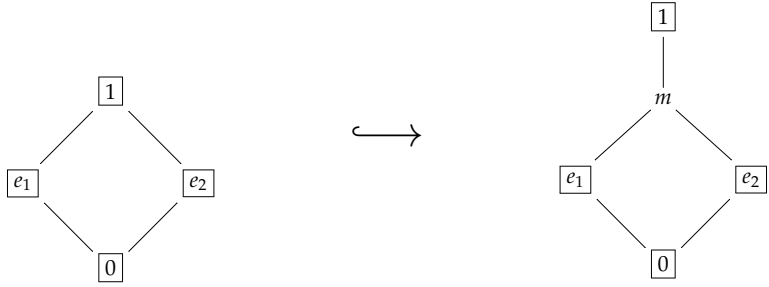
\[\scalebox{0.75}{$
\xymatrix{
& & \\ 
& \fbox{$1$} & \\
\fbox{$e_1$} \ar@{-}[ur] & & \fbox{$e_2$} \ar@{-}[ul] &\\
& \fbox{$0$} \ar@{-}[ul] \ar@{-}[ur] & &
}
$}
\qquad
\scalebox{1.4}{$
\xymatrix{
 \\ 
\lhook\joinrel\longrightarrow
}
$}
\qquad
\scalebox{0.75}{$
\xymatrix{
& & \fbox{$1$} & \\
&  & m \ar@{-}[u] & \\
& \fbox{$e_1$} \ar@{-}[ur] & & \fbox{$e_2$} \ar@{-}[ul] \\
& & \fbox{$0$} \ar@{-}[ul] \ar@{-}[ur] &
}
$}\]
being coextensive would result in a non-trivial product decomposition of the five-element bounded Heyting semilattice on the right. Recall from \cite{Broodryk2019} that a category $\C$ with finite products is \emph{left-coextensive} if the functor $(*)$ is fully faithful, which in turn is equivalent to the requirement that for any two morphisms $f:A \to X$ and $g:B \to Y$ the squares
\[
\xymatrix{
A \ar[d]^f & \ar[l] A \times B \ar[r] \ar[d]^{f\times g} & B \ar[d]^g \\
X & \ar[l]X \times Y \ar[r] & Y
}
\]
are pushouts. 

\begin{proposition} \label{proposition: left-coextensive}
The category $\BoundedHoop$ is left-coextensive. 
\end{proposition}
\begin{proof}
This is a consequence of Lemma~\ref{lemma: idempotent pushout}; given two homomorphisms $f:A \to X$ and $g:B \to Y$, the morphism $f\times g$ sends $(1,0)$ to $(1,0)$ and $\rho_{(1,0)}^{X\times Y}$ is isomorphic to the canonical product projection. The same argument applies to the second projection.
\end{proof}

\begin{proposition} \label{proposition: coextensive morphism}
For any morphism $f:X \to A$ in $\BoundedHoop$, we have that $f$ is coextensive if and only if for any weakly complementary central elements $e_1$ and $e_2$ of $X$, if $e_1 \vee e_2 = 1$ then $f(e_1) \vee f(e_2) = 1$.   
\end{proposition}

\begin{proof}
Suppose first that $f$ is coextensive, and let $e_1$ and $e_2$ be weakly complementary central elements of $X$ with $e_1 \vee e_2 = 1$. It then follows that in the diagram
\[
\xymatrix{
[e_1 \to X] \ar[d] & X \ar[r] \ar[l] \ar[d]^f & [e_2 \to X] \ar[d]\\ 
[f(e_1) \to A] & A \ar[r] \ar[l] & [f(e_2) \to A]
}
\]
the top row is a product diagram by Lemma~\ref{lemma: weakly complementary central}, the squares are pushouts by Lemma~\ref{lemma: idempotent pushout}, so that the bottom row is a product diagram, and hence $a \mapsto (f(e_1) \to a, f(e_2) \to a)$ is an isomorphism. Under this map, $f(e_1)$ gets mapped to $(1, f(e_1))$ and $f(e_2)$ gets mapped to $(f(e_2), 1)$, so that $f(e_1) \vee f(e_2) = 1$. 

Conversely, suppose that $f$ has the stated preservation property, and let $X_1\leftarrow X\to X_2$ be any product diagram. Via the induced isomorphism $X\cong X_1\times X_2$, let $e_1$ and $e_2$ correspond to $(1,0)$ and $(0,1)$, respectively. These elements are weakly complementary and satisfy $e_1\vee e_2=1$, so $f(e_1)\vee f(e_2)=1$. Lemmas~\ref{lemma: weakly complementary central} and~\ref{lemma: idempotent pushout} show that the pushouts of $f$ along the two product projections form a product. Thus $f$ satisfies the first condition for coextensivity; the second follows from Proposition~\ref{proposition: left-coextensive}. Therefore $f$ is coextensive.
\end{proof} 

\begin{remark}\label{remark: Pierce}
Bounded hoops form a Pierce variety: the term
\[
U(x,y,z,w)=(z\to y)\cdot(w\to x)
\]
satisfies $U(x,y,0,1)=x$ and $U(x,y,1,0)=y$. A Pierce variety is coextensive if and only if complementary central pairs are preserved by every homomorphism \cite[Theorem~1]{Botero2021}. As already noticed above, by Lemma~\ref{lemma: weakly complementary central}, a weakly complementary central pair in a bounded hoop is a complementary central pair precisely when its elements admit a join and that join is $1$. Thus Proposition~\ref{proposition: coextensive morphism} is the stability-by-complements condition in this setting.
\end{remark}

A variety has the Fraser--Horn property if every congruence on a direct product is a product congruence \cite{FraserHorn1970}. By \cite[Proposition~7.2]{Hoefnagel2020}, together with \cite[Proposition~3.30]{HoefnagelTheart2025}, this is equivalent to every surjective homomorphism in the variety being coextensive. Since both $\Hoop$ and $\BoundedHoop$ admit the majority term
\[
m(x,y,z)=(x\sqcup y)\wedge(x\sqcup z)\wedge(y\sqcup z),
\]
they are congruence-distributive (equivalently, majority categories in the sense of \cite{Hoefnagel2019MajorityCategories}) and hence have the Fraser--Horn property. The conclusion below also follows directly from Proposition~\ref{proposition: coextensive morphism}.
\begin{corollary} \label{cor: surjective_coextensive}
Every surjective homomorphism in $\BoundedHoop$ is coextensive. 
\end{corollary}
\begin{proof}
Let $f \colon X \to A$ be a surjective homomorphism in $\BoundedHoop$ and let $e_1, e_2$ be weakly complementary central elements in $X$ such that $e_1 \vee e_2 = 1$. We must show that $f(e_1) \vee f(e_2) = 1$ in $A$. Let $y \in A$ be any upper bound of $f(e_1)$ and $f(e_2)$. By surjectivity, there exists $x \in X$ such that $f(x) = y$. Thus, $1 = f(e_1) \to f(x) = f(e_1 \to x)$ and similarly $1 = f(e_2 \to x)$. Let $z = (e_1 \to x) \wedge (e_2 \to x)$. Since $e_i\to(-)$ preserves binary meets, we have $e_1 \to z=e_1\to x$ and, similarly, $e_2 \to z=e_2\to x$. Since $e_1 \vee e_2 = 1$, Lemma~\ref{lemma: weakly complementary central} implies the induced map into the product is injective. Because $z$ and $x$ map to the same elements, we have $z = x$. But $f(z) = f(e_1 \to x) \wedge f(e_2 \to x) = 1$, thus $y = f(x) = 1$. Consequently, $f(e_1) \vee f(e_2) = 1$, and by the preceding proposition, $f$ is coextensive.
\end{proof}

\begin{corollary} \label{cor: subvariety_coextensive}
If $\V$ is a subvariety of $\BoundedHoop$ where $e_1 \vee e_2 = 1$ for any weakly complementary central elements $e_1$ and $e_2$ of any algebra, then $\V$ is coextensive. 
\end{corollary}
\begin{proof}
If $f:X\to A$ is a morphism in $\V$, the image under $f$ of any weakly complementary pair is again weakly complementary. By the hypothesis on $A$, its join is $1$, so Proposition~\ref{proposition: coextensive morphism} applies. The canonical quotient hoops occurring in its proof are homomorphic images and therefore remain in $\V$.
\end{proof}

\begin{corollary}\label{cor: term_coextensive}
If $\V$ is a subvariety of $\BoundedHoop$ such that $\V \models (t \approx 1)$ where $t(x, y)$ is the binary term \[t(x, y) = (((x \to \neg x) \to y) ((\neg x \to x) \to y)) \to y,\] then $\V$ is coextensive.
\end{corollary}
\begin{proof}
Let $X \in \V$ and let $e_1$ and $e_2$ be weakly complementary central elements of $X$. Let $y \in X$ be such that $e_1, e_2 \leq y$. Then we have that 
\begin{align*}
t(e_1, y) &= \big(((e_1 \to \neg e_1) \to y)\cdot  ((\neg e_1 \to e_1) \to y)\big) \to y \\
&= \big(((e_1 \to e_2) \to y)\cdot  ((e_2 \to e_1) \to y)\big) \to y \\
&= \big((e_2 \to y)\cdot  (e_1 \to y)\big) \to y \\
&= (1\cdot 1) \to y \\
&= 1 \to y\\
&= y
\end{align*}
Since $\V \models (t \approx 1)$, it means that we should have $t(e_1, y) = y = 1$. Hence, $e_1 \vee e_2 = 1$ so that by Corollary \ref{cor: subvariety_coextensive} we can deduce that $\V$ is coextensive.
\end{proof}

\begin{corollary} \label{cor: bh_vee_coextensive}
The variety $\BoundedHoop_\vee$ is coextensive.
\end{corollary}
\begin{proof}
In the variety $\BoundedHoop_\vee$, the join is given by the pseudojoin $\sqcup$, by definition, so that by the Corollary \ref{cor: subvariety_coextensive} it follows directly that $\BoundedHoop_\vee$ is coextensive. Alternatively, one can use the fact that the variety of $\BoundedHoop_\vee$ is the variety of $\BoundedHoop$ satisfying the identity
\[\big((x\to y) \to z\big)\cdot \big((y \to x) \to z\big) \to z \approx 1 \] as stated in Theorem 1.6 in \cite{AglianoFerreirimMontagna2007BasicHoops}. This allows us to deduce the same result using Corollary \ref{cor: term_coextensive}.
\end{proof}

\begin{remark}
The variety $\Hoop_\vee$ is precisely the variety of basic hoops \cite[Proposition~3.1 and Theorem~3.3]{Ferreirim2000ShortNote}. Likewise, $\BoundedHoop_\vee$ is the variety of bounded basic hoops and is term-equivalent to the variety of BL-algebras \cite[Section~2]{AglianoMontagna2003VarietiesBL}. Hence Corollary~\ref{cor: bh_vee_coextensive} also recovers the coextensivity of BL-algebras.
\end{remark}

The category $\mathbf{MV}$ was shown to be coextensive in \cite[Proposition~5.6]{MarraMenni2024}; the corollary below recovers this fact.
\begin{corollary} \label{cor: mv_coextensive}
The variety $\mathbf{MV}$ of MV-algebras is coextensive.
\end{corollary}
\begin{proof}
MV-algebras are term-equivalent to bounded Wajsberg hoops \cite{FontRodriguezTorrens1984,AglianoMontagna2003VarietiesBL}, i.e.\ the variety of hoops satisfying the additional identity:
\[
(x\to y)\to y= (y\to x)\to x.
\]
Bounded Wajsberg hoops form a subvariety of $\BoundedHoop_\vee$ and hence satisfy the hypothesis of Corollary~\ref{cor: subvariety_coextensive}. Thus this subvariety is coextensive, and the result transfers across the term-equivalence.
\end{proof}

\begin{corollary} \label{cor: heyt_coextensive}
The category $\mathbf{Heyt}$ of Heyting algebras is coextensive. 
\end{corollary}
\begin{proof}
Heyting algebras may be identified with idempotent bounded hoops whose natural order admits binary joins, with morphisms preserving those joins \cite{BlokFerreirim2000StructureHoops}. The two elements arising from a product decomposition are Boolean complements, and Heyting homomorphisms preserve this property. Moreover, for any complemented element $e$ one has $e\to x=\neg e\vee x$. Hence $\rho_e$ preserves joins and $[e\to X]$ is again a Heyting algebra. Consequently, the pushouts of Lemma~\ref{lemma: idempotent pushout} and the product argument of Proposition~\ref{proposition: coextensive morphism} take place within $\mathbf{Heyt}$.
\end{proof}

For the next result, we recall from \cite{BlokFerreirim2000StructureHoops} the definition of the ordinal sum of two hoops, $A$ and $B$, which we denote by $A \oplus B$. Without loss of generality, let us suppose $A\cap B =\{1\}$. Then $A\oplus B$ is the algebra with underlying set $A\cup B$, top element $1$ and operations:
\[
x\to y = \begin{aligned}
\begin{cases}
    x\to^A y\qquad\text{for }x,y\in A\,,\\
    x\to^B y\qquad\text{for }x,y\in B\,,\\
    y\qquad\text{for }x\in B\,,y\in A\setminus\{1\}\,,\\
    1\qquad \text{for }x\in A\setminus \{1\}\,,y\in B\,,
\end{cases}    
\end{aligned}
\]
and
\[
x\cdot y = \begin{aligned}
\begin{cases}
    x\cdot^A y\qquad\text{for }x,y\in A\,,\\
    x\cdot^B y\qquad\text{for }x,y\in B\,,\\
    y\qquad\text{for }x\in B\,,y\in A\setminus \{1\}\,,\\
    x\qquad \text{for }x\in A\setminus \{1\}\,,y\in B\,,
\end{cases}    
\end{aligned}
\]
The upper summand $B$ is a filter of $A\oplus B$. Moreover, if $B$ is subdirectly irreducible, then $A\oplus B$ is subdirectly irreducible \cite{BlokFerreirim2000StructureHoops}.

\begin{corollary}\label{cor: subdirectlyirreducible_coextensive}
If $\V$ is a subvariety of $\BoundedHoop$ where every subdirectly irreducible member is of the form $\2 \oplus S$ with $S \in \Hoop$, then $\V$ is coextensive.  
\end{corollary}
\begin{proof}
Consider elements $x, y$ of an algebra $X$ in $\V$. Let $r: X \to \prod_{i\in I}X_i$ be a subdirect representation of $X$. Since $x \cdot \neg x = 0$, we have $\pi_i(x)\cdot \pi_i(\neg x) = 0$ for all $i \in I$. As $X_i$ is of the form $\2 \oplus S_i$, this implies that $\pi_i(x) = 0$ or $\pi_i(\neg x) = 0$. Therefore, we have that \[((\pi_i(x) \to \pi_i(\neg x)) \to \pi_i(y))\cdot  ((\pi_i(\neg x) \to \pi_i(x)) \to \pi_i(y)) \to \pi_i(y) = 1.\] Since the subdirect representation is injective, we have that \[((x \to \neg x) \to y)\cdot ((\neg x \to x) \to y) \to y = 1.\] The result now follows from Corollary~\ref{cor: term_coextensive}.
\end{proof}

\begin{remark}
    For every nontrivial bounded hoop $A$, the following conditions are equivalent:
    \begin{enumerate}[label=(\roman*)]
    \item $A\cong\2 \oplus S$ for some $S \in \Hoop$.
    \item For all $x, y \in A$, if $xy = 0$ then $x=0$ or $y=0$.
    \item For all $x \in A$, if $x \neq 0$ then $x \to 0 = 0$.
    \end{enumerate}
    Indeed, $xy=0$ if and only if $y\leqslant x\to0$, so (ii) and (iii) are equivalent. Under these conditions, $S=A\setminus\{0\}$ is a subhoop and the mixed operations are precisely those of $\2\oplus S$.
\end{remark}

Idempotent hoops are precisely implicative (also called Brouwerian or Heyting) semilattices, and idempotent bounded hoops are precisely bounded implicative semilattices \cite{BlokFerreirim2000StructureHoops}. We will denote the category of bounded implicative semilattices by $\BoundedHeytingSemilattices$, and we can characterise the coextensive subvarieties of $\BoundedHeytingSemilattices$ with the following result. 

\begin{proposition}
    Let $\V$ be a subvariety of $\BoundedHeytingSemilattices$. Then the following are equivalent:
    \begin{enumerate}[label=(\roman*)]
    \item $\V$ is coextensive
    \item $\2^2 \oplus \2 \notin \V$
    \item $\V \models (t \approx 1)$ where $t(x, y)$ is the binary term 
    \begin{align*}
    t(x, y) &= ((x \to \neg x) \to y)\cdot ((\neg x \to x) \to y) \to y\\
    &= ((\neg x \to y) \wedge (\neg \neg x \to y)) \to y.
    \end{align*}
    Here the second equality uses $x\to\neg x=\neg x$ and $\neg x\to x=\neg\neg x$, valid in every idempotent bounded hoop.
    \item All subdirectly irreducible members of $\V$ are of the form $\2 \oplus S$ for some $S \in \HeytingSemilattices$.
    \end{enumerate}
\end{proposition}
\begin{proof}
$(i) \Rightarrow (ii).$ Suppose that $C=\2^2 \oplus \2$ belongs to $\V$. Then $\2^2$ is a subalgebra of $C$, and, writing $e_1=(1,0)$ and $e_2=(0,1)$, we have the embedding
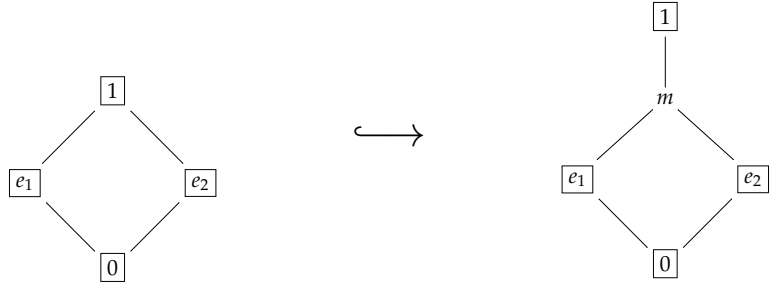
\[\scalebox{0.75}{$
\xymatrix{
& & \\ 
& \fbox{$1$} & \\
\fbox{$e_1$} \ar@{-}[ur] & & \fbox{$e_2$} \ar@{-}[ul] &\\
& \fbox{$0$} \ar@{-}[ul] \ar@{-}[ur] & &
}
$}
\qquad
\scalebox{1.4}{$
\xymatrix{
 \\ 
\lhook\joinrel\longrightarrow
}
$}
\qquad
\scalebox{0.75}{$
\xymatrix{
& & \fbox{$1$} & \\
&  & m \ar@{-}[u] & \\
& \fbox{$e_1$} \ar@{-}[ur] & & \fbox{$e_2$} \ar@{-}[ul] \\
& & \fbox{$0$} \ar@{-}[ul] \ar@{-}[ur] &
}
$}\]
The canonical pushouts of this embedding along the two projections are the quotient maps of Lemma~\ref{lemma: idempotent pushout}; their codomains are isomorphic to $\2$ and belong to $\V$. The induced map $C\to\2^2$ identifies $m$ and $1$, so these pushouts do not form a product. Thus the embedding is not coextensive in $\V$, and hence $\V$ is not coextensive.

$(ii) \Rightarrow (iii).$ We prove the contrapositive. Suppose that $\V\not\models(t\approx1)$, and choose $X\in\V$ and $x,y\in X$ such that $m=t(x,y)\neq1$. Put $a=\neg x$, $b=\neg\neg x$, and $q=(a\to y)\wedge(b\to y)$, so that $m=q\to y$. We have $a\wedge b=0$, while $q\leqslant a\to y$ and $q\leqslant b\to y$ give $a,b\leqslant m$. Since $m\neq1$, the elements $a$ and $b$ are nonzero, proper, and incomparable; in particular,
\[
0<a,b<m<1.
\]
Moreover,
\[
\neg a=b,\quad \neg b=a,\quad a\to b=b,\quad b\to a=a,
\quad m\to a=a,\quad m\to b=b,\quad \neg m=0.
\]
Together with the displayed order, these identities show that $\{0,a,b,m,1\}$ is closed under $\wedge,\to,0,1$ and is isomorphic to $\2^2\oplus\2$. Since $\V$ is closed under subalgebras, $\2^2\oplus\2\in\V$.

$(iii) \Rightarrow (iv).$ Again we prove the contrapositive. Suppose that a subdirectly irreducible algebra $A\in\V$ is not of the form $\2\oplus S$. By the preceding remark, there exists $x\neq0$ such that $\neg x\neq0$. We have $\neg x\wedge\neg\neg x=0$ and $\neg\neg x\geqslant x\neq0$, so $\neg x$ and $\neg\neg x$ are incomparable. A subdirectly irreducible implicative semilattice has a greatest element $\mu<1$ \cite[Remark~2.4]{BezhanishviliEtAl2021}. Consequently, $\neg x,\neg\neg x<\mu$, and hence
\[
t(x,\mu)=((\neg x\to\mu)\wedge(\neg\neg x\to\mu))\to\mu
=(1\wedge1)\to\mu=\mu\neq1.
\]
Thus $\V\not\models(t\approx1)$.

$(iv) \implies (i).$ This follows from Corollary~\ref{cor: subdirectlyirreducible_coextensive}.
\end{proof}

\bibliographystyle{plain}
\bibliography{references}

\end{document}